\documentclass[11pt]{amsart}
\usepackage{graphicx}
\usepackage{amssymb}
\usepackage{amsmath}

\theoremstyle{plain}
\newtheorem{theorem}{Theorem}[section]
\newtheorem{lemma}[theorem]{Lemma}
\newtheorem{claim}[theorem]{Claim}

\newtheorem{corollary}[theorem]{Corollary}
\newtheorem{corollary*}[theorem]{Corollary}
\newtheorem{proposition}[theorem]{Proposition}

\theoremstyle{definition}
\newtheorem{definition}[theorem]{Definition}

\theoremstyle{remark}
\newtheorem*{remark}{Remark}

\def\d{\displaystyle}
\def\z{\hat{\zeta}}
\def\tz{\zeta}
\def\C{\mathcal{C}}
\def\O{\text{ Ord}}
\def\lo{\underline{o}}
\def\bigz{\mathbb{Z}}
\newcommand\gl{\mathfrak{gl}}
\def\comp{\circ}

\begin{document}

\normalsize

\title{Calculation of local formal Fourier transforms}
\author{Adam Graham-Squire}
\date{}

 \begin{abstract} We calculate the local Fourier transforms for connections on the formal punctured disk, reproducing the results of J. Fang \cite{fang} and C. Sabbah \cite{sabbah} using a different method. Our method is similar to Fang's, but more direct.
 \end{abstract}

\maketitle

\section{Introduction}\label{intro}

In \cite{bloch}, S.~Bloch and H.~Esnault introduced the local Fourier transforms for connections on the formal punctured disk. In \cite{garcialopez}, R.~Garcia Lopez found similar results to \cite{bloch} using a different method.  Neither \cite{bloch} nor \cite{garcialopez} gave explicit calculations for the local Fourier transforms, however.
Explicit formulas were proved by J.~Fang \cite{fang} and C.~Sabbah \cite{sabbah}. Interestingly, the calculations rely
on different ideas: the proof of \cite{fang} is more algebraic, while \cite{sabbah} uses geometric methods.

In this paper, we provide yet another proof of these formulas. Our approach is closer to Fang's, but more straightforward. In order to
calculate a particular local Fourier transform, one must ascertain the `canonical form' of the local Fourier transform of a given connection.
This amounts to constructing an isomorphism between two connections (on a punctured formal disk). In \cite{fang}, this is done
by writing matrices of the connections with respect to certain bases. We work with operators directly, using techniques
described by D. Arinkin in \cite[Section 7]{dima}.

\subsection*{Acknowledgements}\label{subsec ack} I am very grateful to my advisor Dima Arinkin for many helpful discussions and his consistent encouragement of this work.

\section{Definitions and Conventions}\label{sec def}

We fix a ground field $\Bbbk$, which is assumed to be algebraically closed of characteristic zero.

\subsection{Connections on formal disks}\label{subsec conn}
Consider the field of formal Laurent series $K=\Bbbk((z))$.

\begin{definition} \label{connection} Let $V$ be a finite-dimensional vector space over $K$. A \emph{connection} on $V$ is a $\Bbbk$-linear operator
$\nabla\colon V\to V$ satisfying the Leibniz identity:
\[\nabla(fv)=f\nabla(v)+\frac{df}{dz}v\]
for all $f\in K$ and $v\in V$. A choice of basis in $V$ gives an isomorphism $V\simeq K^n$; we can then write $\nabla=\nabla_z$ as $\frac{d}{dz}+A$, where
$A=A(z)\in\gl_n(K)$ is the \emph{matrix} of $\nabla$ with respect to this basis.
\end{definition}

We write $\mathcal{C}$ for the category of vector spaces with connections over $K$. Its objects are pairs
$(V, \nabla)$, where $V$ is a finite-dimensional $K$-vector space and $\nabla\colon V\to V$ is a connection. Morphisms between
$(V_1,\nabla_1)$ and $(V_2,\nabla_2)$ are $K$-linear maps $\phi\colon V_1\to V_2$ that are \emph{horizontal} in the sense that
$\phi\nabla_1=\nabla_2\phi$.

 We summarize below some well-known
properties of connections on formal disks. The results go back to Turrittin \cite{turrittin} and Levelt \cite{levelt};
more recent references include \cite{varadar}, \cite[Sections 5.9 and 5.10]{beilinson}, \cite{malgrange}, and \cite{vander}.

Let $q$ be a positive integer and consider the field $K_q=\Bbbk((z^{1/q}))$. Note that $K_q$ is the unique extension of $K$ of degree $q$.
For every $f\in K_q$, we define an object $E_f\in\mathcal{C}$ by
\[E_f=E_{f,q}=\left(K_q,\frac{d}{dz}+z^{-1}f\right).\]

In terms of the isomorphism class of an object $E_f$, the reduction procedures of \cite{turrittin} and \cite{levelt} imply that we need only consider $f$ in the quotient
\begin{equation}\label{conn equation}
\Bbbk((z^{1/q}))\Big/\left(z^{1/q}\Bbbk[[z^{1/q}]]+\frac{1}{q}\mathbb{Z}\right)
\end{equation}
where $\Bbbk[[z]]$ denotes formal \emph{power} series.

Let $R_q$ (we write $R_q(z)$ when we wish to emphasize the local coordinate) be the set of orbits for the action of the Galois group $\mathrm{Gal}(K_q/K)$ on the quotient\label{conn quotient}. Explicitly, the
Galois group is identified with the group of degree $q$ roots of unity $\eta\in \Bbbk$; the action on $f\in R_q$ is by 
$f(z^{1/q})\mapsto f(\eta z^{1/q})$. Finally, let $R^\comp_q\subset R_q$ denote the set of $f\in R_q$ that cannot be represented 
by elements of $K_r$ for any $0<r<q$.

\begin{remark}
  $R^\comp_q$ can alternatively be described as the locus of $R_q$ where Gal($K_q/K$) acts freely.
\end{remark}

The following proposition lists some well-known facts about the objects $E_f$.  The proofs of the different parts of the proposition are either straightforward or common in the literature, and are thus omitted.

\begin{proposition}\label{prop} \mbox{}
\begin{enumerate}
\item \label{prop1} The isomorphism class of $E_f$ depends only on the orbit of the image of $f$ in $R_q$.
\item \label{prop2} $E_f$ is irreducible if and only if the image of $f$ in $R_q$ belongs to $R^\comp_q$. As $q$ and $f$ vary,
we obtain a complete list of irreducible objects of $\mathcal{C}$.
\item \label{prop3}Every $E\in\mathcal{C}$ can be written as
\[E\simeq\bigoplus_i(E_{f_i,q_i}\otimes J_{m_i}),\]
where the $E_{f,q}$ are irreducible and $J_m=(K^m,\frac{d}{dz}+z^{-1}N_m)$, with $N_m$ representing the nilpotent Jordan block of size $m$.
\end{enumerate}
\end{proposition}

\begin{remark}
Proposition \ref{prop} \eqref{prop3} is particularly useful because it allows us to reduce the calculation of the local Fourier transform of $E\in\mathcal{C}$ to looking at the calculation on $E_f$.  A precise statement is found in Corollary \ref{corollary}.
\end{remark}

\subsection{Local Fourier transforms}\label{subsec LFT}

 Sometimes it is useful to keep track of the choice of local coordinate for $\mathcal{C}$.  To stress the coordinate, we write $\mathcal{C}_0$ to indicate the coordinate $z$ at the point zero and $\mathcal{C}_{\infty}$ to indicate the coordinate $\zeta=\frac{1}{z}$ at the point at infinity.  Note that $\mathcal{C}_0$ and $\mathcal{C}_{\infty}$ are both isomorphic to $\mathcal{C}$, but not canonically.  We also let $\mathcal{C}_{\infty}^{<1}$ (respectively $\mathcal{C}_{\infty}^{>1}$) denote the full subcategory of $\mathcal{C}_{\infty}$ of connections whose irreducible components all have slopes less than one (respectively greater than one); that is, $E_f$ such that $-1< \text{ord}(f)$ (respectively $-1>\text{ord}(f)$).

\begin{definition}
We define the local Fourier transforms $\mathcal{F}^{(0,\infty)}$, $\mathcal{F}^{(\infty,0)}$ and $\mathcal{F}^{(\infty,\infty)}$ using the relations given in \cite[Propositions 3.7, 3.9 and 3.12]{bloch} while following the convention of \cite[Section 2.2]{dima}.  We let the Fourier transform coordinate of $z$ be $\hat{z}$, with $\z=\frac{1}{\hat{z}}$. Let $E=(V,\nabla_z)\in\mathcal{C}_0$ such that $\nabla_z$ has no horizontal sections, thus $\nabla_z$ is invertible.  The following is a precise definition for $\mathcal{F}^{(0,\infty)}(E)$; the other local Fourier transforms can be defined analogously and thus precise definitions are omitted.  Consider on $V$ the $\Bbbk$-linear operators
\begin{equation}\label{fourierdef}
	\z=-\nabla_z^{-1}\colon V\to V \text{ and } \hat{\nabla}_{\z}=-\z^{-2}z\colon V \to V.
\end{equation}
As in \cite{dima}, $\z$ extends to define an action of $\Bbbk((\z))$ on $V$ and dim$_{\Bbbk((\z))}V<\infty$. We write $V_{\z}$ to indicate that we are considering $V$ as a $\Bbbk((\z))$-vector space. Then $\hat{\nabla}_{\z}$ is a connection, and the $\Bbbk((\z))$-vector space $V_{\z}$ with connection $\hat{\nabla}_{\z}$ is denoted by
\[ \mathcal{F}^{(0,\infty)}(E):=(V_{\z},\hat{\nabla}_{\z})\in \mathcal{C}_{\infty}^{<1},\]
which defines the functor $\mathcal{F}^{(0,\infty)}\colon \mathcal{C}_0 \to \mathcal{C}_{\infty}^{<1}$.
\end{definition}

Given the conventions above, we can express the other local Fourier transforms by the functors
$$\mathcal{F}^{(\infty,0)}\colon \mathcal{C}_{\infty}^{<1} \to \mathcal{C}_0 \text{   and   }\mathcal{F}^{(\infty,\infty)}\colon \mathcal{C}_{\infty}^{>1}\to \mathcal{C}_{\infty}^{>1}.$$

If one considers only the full subcategories of $\mathcal{C}_0$ and $\mathcal{C}_{\infty}^{<1}$ of connections with no horizontal sections, the functors $\mathcal{F}^{(0,\infty)}$ and $\mathcal{F}^{(\infty,0)}$ define an equivalence of categories.  Similarly, $\mathcal{F}^{(\infty,\infty)}$ is an auto-equivalence of the subcategory $\mathcal{C}_{\infty}^{>1}$ \cite[Propositions 3.10 and 3.12]{bloch}.

\section{Statement of theorems}\label{state thms}

Let $s$ be a nonnegative integer and $r$ a positive integer.

\subsection{Calculation of $\mathcal{F}^{(0,\infty)}$}\label{subsec calc1}

\begin{theorem}\label{thm1} Let $f\in R^\comp_r(z)$ with ord$(f)=-s/r$ and $f\neq 0$.  Then $E_f\in \C_0$ and
\[\mathcal{F}^{(0,\infty)}(E_f)\simeq E_g,\]
where $g\in {R^\comp_{r+s}(\z)}$ is determined by the following system of equations:
\begin{equation}\label{zisyseq1}f=-z\hat{z}
\end{equation}
\begin{equation}\label{zisyseq2} g=f+\frac{s}{2(r+s)}
\end{equation}
\end{theorem}

\begin{remark} Recall that $\d{\z=\frac{1}{\hat{z}}}$.  We determine $g$ using \eqref{zisyseq1} and \eqref{zisyseq2} as follows.
First, using \eqref{zisyseq1} we express $z$ in terms of $\z^{1/(r+s)}$.  We then substitute that expression for $z$ into \eqref{zisyseq2} and solve to get an expression for $g(\z)$ in terms of $\z^{1/(r+s)}$.

When we use \eqref{zisyseq1} to write an expression for $z$ in terms of $\z^{1/(r+s)}$, the expression is not unique since we must make a choice of a root of unity.  More concretely, let $\eta$ be a primitive $(r+s)^{\text{th}}$ root of unity.  Then replacing $\z^{1/(r+s)}$ with $\eta\z^{1/(r+s)}$ in our equation for $z$ will yield another possible expression for $z$.  This choice will not affect the overall result, however, since all such expressions will lie in the same Galois orbit.  Thus by Proposition \ref{prop} \eqref{prop1}, they all correspond to the same connection.
\end{remark}

\begin{corollary}\label{corollary}

Let $E$ be an object in $\mathcal{C}$.  By Proposition \ref{prop} (3), let $E$ have decomposition
$\d{E\simeq \bigoplus_i \bigg(E_{f_i}\otimes J_{m_i}\bigg).}$
Then
$$\mathcal{F}^{(0, \infty)}(E)\simeq \bigoplus_i \bigg(E_{g_i}\otimes J_{m_i}\bigg)$$for $E_{g_i}=\mathcal{F}^{(0, \infty)}(E_{f_i})$ as defined in Theorem \ref{thm1}.
\end{corollary}

\begin{proof}[Sketch of Proof]
  $E_f\otimes J_m$ is the unique indecomposable object in $\mathcal{C}$ formed by $m$ successive extensions of $E_f$.  Since we have an equivalence of categories, we only need to know how $\mathcal{F}^{(0, \infty)}$ acts on $E_f$.  This is given by Theorem \ref{thm1}.
\end{proof}

\subsection{Calculation of $\mathcal{F}^{(\infty,0)}$}\label{subsec calc2}

\begin{theorem}\label{thm2}
	 Let $f\in R^\comp_r(\zeta)$ with ord$(f)=-s/r$, $s<r$, and $f\neq 0$. Then $E_f\in \C^{<1}_{\infty}$ and
\[\mathcal{F}^{(\infty, 0)}(E_{f})\simeq E_{g},\]
where $g\in R^\comp_{r-s}(\hat{z})$ is determined by the following system of equations:
\begin{equation}\label{izsyseq1}f=z\hat{z}
\end{equation}
\begin{equation}\label{izsyseq2} g=-f+\frac{s}{2(r-s)}
\end{equation}

\end{theorem}

\begin{remark}
We determine $g$ from \eqref{izsyseq1} and \eqref{izsyseq2} as follows.  First, we use \eqref{izsyseq1} to express $\zeta$ in terms of $\hat{z}^{1/(r-s)}$.  We then substitute this expression into \eqref{izsyseq2} to get an expression for $g(\hat{z})$ in terms of $\hat{z}^{1/(r-s)}$.
\end{remark}

\subsection{Calculation of $\mathcal{F}^{(\infty,\infty)}$}\label{subsec calc3}

\begin{theorem}\label{thm3}
    Let $f\in  R^\comp_r(\zeta)$ with ord$(f)=-s/r$ and $s>r$. Then $E_f\in \C^{>1}_{\infty}$ and
\[\mathcal{F}^{(\infty, \infty)}(E_{f})\simeq E_{g},\]
where $g\in R^\comp_{s-r}(\z)$ is determined by the following system of equations:
\begin{equation}\label{iisyseq1}f=z\hat{z}
\end{equation}
\begin{equation}\label{iisyseq2} g=-f+\frac{s}{2(s-r)}
\end{equation}

\end{theorem}

\begin{remark}
  We determine $g$ from \eqref{iisyseq1} and \eqref{iisyseq2} as follows.  First, we use \eqref{iisyseq1} to express $\zeta$ in terms of $\z^{1/(s-r)}$.  We then substitute this expression into \eqref{iisyseq2} to get an expression for $g(\z)$ in terms of $\z^{1/(s-r)}$.
\end{remark}

\section{Proof of Theorems}\label{sec proofs}

\subsection{Outline of Proof of Theorem \ref{thm1}}\label{F proof subsec} We start with the operators given in \eqref{fourierdef}, viewing them as equivalent operators on $K_r$.  We wish to understand how the operator $\hat{\nabla}_{\z}$ acts in terms of the operator $\z$.  The proof is broken into two cases, depending on the type of singularity. In the case of regular singularity, we have ord$(f)=0$, and the proof is fairly straightforward.  In the irregular singularity case where ord$(f)<0$, the proof hinges upon defining a fractional power of an operator, which is done in Lemma \ref{abratlemma}.  Lemma \ref{abratlemma} is the heavy lifting of the proof; the remaining portion is just calculation to extract the appropriate constant term (see remark below) from the expression given by Lemma \ref{abratlemma}.

\begin{remark} We give a brief explanation regarding the origin of the system of equations found in Theorem \ref{thm1}.  Consider the expressions given in \eqref{fourierdef}.  Suppose we were to make a \lq\lq naive" local Fourier transform over $K_r$ by defining $\nabla_z=z^{-1}f(z)$ and $\hat{\nabla}_{\z}=\z^{-1}g(\z)$; in other words, as in Definition \ref{connection} but without the differential parts.  Then from the equation $-(z^{-1}f)^{-1}=\z$ we conclude
\begin{equation}\label{syseq1nd}
f=-z\hat{z}.
\end{equation}Similarly, from $-\z^{-2}z=\z^{-1}g$ we find
$-\hat{z} z=g$, which when combined with \eqref{syseq1nd} gives
\begin{equation}\label{syseq2nd}
f=g.
\end{equation}When one incorporates the differential parts into the expressions for $\nabla_{z}$ and $\hat{\nabla}_{\z}$, one sees that the system of equations \eqref{syseq1nd} and \eqref{syseq2nd} nearly suffices to find the correct expression for $g(\z)$, only a constant term is missing.  This constant term arises from the interplay between the differential and linear parts of $\nabla_{z}$, and we wish to derive what the value of it is.  Similar calculations can be carried out to justify the systems of equations for Theorems \ref{thm2} and \ref{thm3}.
\end{remark}

\subsection{Lemmas}\label{subsec lemmas}

 \begin{definition} Let $A$ and $B$ be $\Bbbk$-linear operators from $K_q$ to $K_q$.
We define Ord($A$) to be
$$\d{\text{Ord}(A)=\inf_{f\in K_q}\big(\text{ord}(Af)-\text{ord}(f)\big), \text{ with Ord}(0):=\infty}$$ and define $\underline{o}(z^k)$ by $$A=B+\underline{o}(z^k) \text{  if and only if  } \text{Ord}(A-B)\geq k.$$
We say that $A$ is a \emph{similitude} if $\text{Ord}(A)=\text{ord}(Af)-\text{ord}(f)$ for \underline{any} $f\in K_q$.
\end{definition}

\begin{lemma}\label{abintlemma}
    Let $A$ and $B$ be $\Bbbk$-linear operators on $K_q$, with the following conditions:  $A$ and $A+B$ are similitudes, and $[A,[B,A]]=0$. Let Ord$(A)=a$, Ord$(B)=b$, and suppose that $a< b$.  Then
    \begin{equation}\label{abint}
        (A+B)^{m}=A^{m}+mA^{(m-1)}B+\frac{m(m-1)}{2}A^{m-2}[B,A] + \underline{o}(z^{a(m-1)+b})
    \end{equation}for all $m\in\mathbb{Z}$.
\end{lemma}
\begin{proof} We first prove that \eqref{abint} holds for $m\geq 0$ using induction. The case $m=0$ is trivial.
Assuming the equation holds for $(A+B)^m$, we have
    \begin{equation*}
    \begin{split}
        (A+B)^{m+1} & = (A+B)^{m}(A+B)\\
         & = A^{m+1}+mA^{m-1}BA+\frac{m(m-1)}{2}A^{m-2}[B,A]A+A^mB + \underline{o}(z^{a(m-1)+b+a})\\
          & = A^{m+1}+(m+1)A^{m}B+mA^{m-1}[B,A]+\frac{m(m-1)}{2}A^{m-1}[B,A] + \underline{o}(z^{am+b})\\
                    & = A^{m+1}+(m+1)A^{m}B+\frac{m(m+1)}{2}A^{m-1}[B,A] + \underline{o}(z^{am+b})
    \end{split}
    \end{equation*}which completes the induction for the nonnegative integers .
 Since $A+B$ is invertible, the expansion
 $$(A+B)^{-1}=A^{-1}-A^{-1}BA^{-1}+A^{-1}BA^{-1}BA^{-1}-\dots$$
 is well-defined. Using that expansion (which verifies the base case $m=-1$), the proof for $m\leq -1$ follows in the same manner as the proof for the nonnegative integers above.  Note that the condition $\O(A^{-1})=-\O(A)$ (which follows from $A$ being a similitude) is necessary for the induction on the negative integers.
\end{proof}

We now wish to use \eqref{abint} to define \emph{fractional} powers of the operator $(A+B)$, given certain operators $A$ and $B$.  We follow the method of \cite[Section 7.1]{dima} to extend the definition, though our goal is more narrow; Arinkin defines powers for all $\alpha\in \Bbbk$, but we only need to define fractional powers $m\in\frac{1}{p}\mathbb{Z}$ for a given nonzero integer $p$.

\begin{lemma}\label{abratlemma}
    Let $A$ and $B$ be the following $\Bbbk$-linear operators on $K_q$: $A=\text{ multiplication by } f= jz^{p/q}+ \underline{o}(z^{p/q})$, $0\neq j \in \Bbbk$, and $B=z^n\frac{d}{dz}$ with $n\neq 0$, $p\neq 0,$ and $q> 0$ all integers.  We have Ord$(A)=\frac{p}{q}$ and  Ord$(B)=n-1$, and we assume that $\frac{p}{q}< n-1$. Then we can choose a $p^{\text{th}}$ root of $(A+B)$, $(A+B)^{1/p}$, such that
\begin{equation*}
        (A+B)^{m}=A^{m}+mA^{(m-1)}B+\frac{m(m-1)}{2}A^{m-2}[B,A] + \underline{o}(z^{(p/q)(m-1)+n-1})
    \end{equation*} holds for all $m\in \frac{1}{p}\mathbb{Z}$ where $(A+B)^{m}=((A+B)^{1/p})^{pm}$.
\end{lemma}

\begin{proof}
    We use the notation found in \cite[Section 7.1]{dima}.  Letting $P=(1/j)(A+B)$ we have $P\colon K_q\to K_q$ is $\Bbbk$-linear of the form
    $$P\left(\sum_{\beta}c_{\beta}z^{\beta/q} \right)=\sum_{\beta}c_{\beta}\sum_{i\geq 0}p_i(\beta)z^{(\beta+i+p)/q}.$$
    Thus $p_0(\beta)=1$ and all $p_i$ are constants or have the form $\beta/q+$constant, so the necessary conditions \cite[Section 7.1, conditions (1) and (2)]{dima} are satisfied.  We can now define $P^m$, and likewise $(A+B)^m=j^mP^m$, for $m=\frac{1}{p}$.
\end{proof}

\subsection{Proof of Theorem \ref{thm1}}\label{subsec proof1}

\begin{proof} From \cite[Proposition 3.7]{bloch} we have the following equations for the local Fourier transform $\mathcal{F}^{(0,\infty)}$:
\begin{equation}\label{benotationzi}
    z=-\z^2\partial_{\z} \text{ and } \partial_z=-\z^{-1}.
\end{equation}

Converting to our notation, we write $\partial_{\z}=\hat{\nabla}_{\z}=\frac{d}{d\z}+\z^{-1}g(\z)$ and $\partial_z=\nabla_{z}=\frac{d}{dz}+z^{-1}f(z)$. Then \eqref{benotationzi} becomes

\begin{equation}\label{mynotationzi}
    z=-\z^2\frac{d}{d\z}-\z g(\z)
\end{equation}
 and
\begin{equation}\label{firsteq}
    \frac{d}{dz}+z^{-1}f(z)=-\z^{-1}.
\end{equation}
Our goal is to use \eqref{firsteq} to write an expression for the operator $z$ in terms of $\z$, at which point we can substitute into \eqref{mynotationzi} to find an expression for $g(\z)$.

\noindent \textbf{Case One}: Regular singularity (ord$(f)=0$).

In this case we have $s=0$ and $r=1$, so $f=\alpha\in \Bbbk\setminus\mathbb{Z}$.  Then \eqref{firsteq} has form $\frac{d}{dz}+\frac{\alpha}{z}=-\z^{-1}$.  But on $K$, the operator $\frac{d}{dz}$ acts on monomials as multiplication by $\frac{n}{z}$ for some $n\in\mathbb{Z}$, and $f\in R^o_r(z)$ means that $\alpha$ is only defined up to a shift by $\mathbb{Z}$. Thus the operator $\frac{d}{dz}+\frac{\alpha}{z}$ acts in the same manner as just $\frac{\alpha}{z}$.  In other words, we can safely ignore the differential part of the operator in the case of a regular singularity.  The remainder of this case follows from the remark below the outline in subsection \ref{F proof subsec}.

\noindent \textbf{Case Two}: Irregular singularity (ord$(f)<0$).

Consider the equation
\begin{equation}\label{firsteq no diff}
z^{-1}f=-\z^{-1},
\end{equation}
 which is \eqref{firsteq} without the differential part, and coincides with \eqref{zisyseq1}.  Equation \eqref{firsteq no diff} can be thought of as an implicit expression for the variable $z$ in terms of $\z$, which one can rewrite as an explicit expression $z=h(\z)\in \Bbbk((\z^{1/(r+s)}))$ for the variable $z$. This is the purely algebraic calculation which in Theorem \ref{thm1} is stated as expressing $z$ in terms of $\z^{1/(r+s)}$.  Note that since there is no differential part in \eqref{firsteq no diff}, $h(\z)$ is \underline{not} the same as the operator $z$.  Since the leading term of $z^{-1}f(z)$ is $az^{-(r+s)/r}$ (for some $a\in\Bbbk$), \eqref{firsteq no diff} implies that $h(\z)=a^{r/(r+s)}(-\z)^{r/(r+s)}+\lo(\z^{r/(r+s)})$.  Using \eqref{firsteq} we find that the operator $z$ will be of the form
\begin{equation}\label{zdefzi}
    z=h(\z)+*(-\z)+\underline{o}(\z)
\end{equation}
where the $*\in\Bbbk$ represents the coefficient that arises from the interplay between the differential and linear parts of $-\z=\nabla_z^{-1}$.  As explained in the outline, we wish to find the value of *.  Let $A=z^{-1}f(z)$ and $B=\frac{d}{dz}$, then $[B,A]=A'=z^{-1}f'-z^{-2}f$. From \eqref{firsteq} we have $-\z=(A+B)^{-1}$, and we apply Lemma \ref{abratlemma} to find
\begin{equation*}
    (-\z)^{\frac{r}{r+s}}=a^{\frac{-r}{r+s}}\left(z+\dots+a^{-1}\left[\frac{-r}{r+s}\left(\frac{\mathbb{Z}}{r}\right)+\frac{-r}{r+s} + \frac{-s}{2(r+s)}\right]z^{\frac{r+s}{r}}+\underline{o}(z^{\frac{r+s}{r}})\right).
\end{equation*}
\begin{remark}  We use the notation $\frac{\mathbb{Z}}{r}$ to represent the operator $z\frac{d}{dz}$.  This notation makes sense, because $z\frac{d}{dz}\colon K_r\to K_r$ acts as $z\frac{d}{dz}(z^{n/r})=\frac{n}{r}(z^{n/r})$ for any $n\in\mathbb{Z}$.
\end{remark}
Also from Lemma \ref{abintlemma} we have
\begin{equation*}
    (-\z)=a^{-1}z^{1+(s/r)}+\underline{o}(z^{1+(s/r)}).
\end{equation*}
The appropriate value for * in \eqref{zdefzi} is the expression that will make the leading term of $*(-\z)$, which will be $*a^{-1}z^{1+(s/r)}$, cancel with $a^{-1}\left[\frac{-\mathbb{Z}}{r+s}+\frac{-r}{r+s} + \frac{-s}{2(r+s)}\right]z^{1+(s/r)}$. Thus we find that
\begin{equation}\label{starzi}
    *=\frac{\mathbb{Z}+r}{r+s}+\frac{s}{2(s+r)}.
\end{equation}

Applying both sides of \eqref{mynotationzi} to $1\in K_r$, and using the fact that $\frac{d}{d\z }(1)=0$, we see that $z=-\z g(\z)$.
 Thus to find the expression for $g$ we simply need to compute the Laurent series in $\z$ given by $(-\z^{-1})z$.  Substituting the expressions from \eqref{zdefzi} and \eqref{starzi} into $(-\z^{-1})z$, we have
\begin{equation*}
    g(\z)=-\z^{-1} h(\z) + \left(\frac{\mathbb{Z}+r}{r+s} +\frac{s}{2(r+s)}\right)+\underline{o}(1).
\end{equation*}  By Proposition \ref{prop}, (1), $E_{g,r+s}$ will be isomorphic to $E_{\dot{g},r+s}$ where
\begin{equation}\label{final g zi}
\dot{g}(\z)=-\z^{-1} h(\z) + \frac{s}{2(r+s)},
\end{equation}
since $g$ and $\dot{g}$ differ only by $\frac{\mathbb{Z}+r}{r+s}\in \frac{1}{r+s}\mathbb{Z}$. From \eqref{firsteq no diff} we have $-\z^{-1} h(\z)=-z\hat{z}=f$, so \eqref{final g zi} matches \eqref{zisyseq2} which completes the proof.
  \end{proof}

\subsection{Proof of Theorem \ref{thm2}}\label{subsec proof 2}

\begin{proof}
This proof is much the same as the proof of Theorem \ref{thm1}, so we only sketch the pertinent details.  From \cite[Proposition 3.9]{bloch}, in our notation we have
\begin{equation}\label{firsteq iz}
    \tz^2\nabla_{\tz}=\hat{z} \text{ and } \tz^{-1}=-\hat{\nabla}_{\hat{z}}
\end{equation}

We wish to write $z=\zeta^{-1}$ in terms of $\hat{z}^{1/(r-s)}$. Consider the equation
\begin{equation}\label{firsteq iz no diff}
\zeta f=\hat{z}
\end{equation}
which is the first equation of \eqref{firsteq iz} without the differential part.  We can think of \eqref{firsteq iz no diff} as an implicit definition for the variable $\zeta$, which we can rewrite as an explicit expression $\zeta=h(\hat{z})=a^{-r/(r-s)}\hat{z}^{r/(r-s)}+\lo(\hat{z}^{r/(r-s)})$.  Letting $A=\tz f(\tz)$, $B=\tz^2\frac{d}{d\tz}$ and $\hat{z}=A+B$, we have $[B,A]=\zeta^2A'$ and the Operator-root Lemma gives
$$\hat{z}^{r/(r-s)}=a^{r/(r-s)}\left(\tz+\dots +a^{-1}\left[\frac{r}{r-s}\left(\frac{\mathbb{Z}}{r}\right)+\frac{s}{2(r-s)}\right]\tz^{1+(s/r)}+ \underline{o}(\tz^{1+(s/r)})\right)$$
and
$$\hat{z}^{(r+s)/(r-s)}=a^{(r+s)/(r-s)}\tz^{1+(s/r)}+ \underline{o}(\tz^{1+(s/r)}).$$
We conclude that the operator $\zeta$ will be
$$\tz=h(\hat{z})+a^{-2r/(r-s)}\left[\frac{-\mathbb{Z}}{r-s}+ \frac{-s}{2(r-s)}\right]\hat{z}^{(r+s)/(r-s)}+ \underline{o}(\hat{z}^{(r+s)/(r-s)}).$$
Inverting the operator $\zeta$, we find
$$\zeta^{-1}=z=h(\hat{z})^{-1}+\left(\frac{\mathbb{Z}}{r-s}+ \frac{s}{2(r-s)}\right)\hat{z}^{-1}+ \underline{o}(\hat{z}^{-1})$$
and it follows that

$$g(\hat{z})=-\hat{z} z=-\hat{z}h(\hat{z})^{-1}+\frac{-\mathbb{Z}}{r-s}+\frac{-s}{2(r-s)}+\underline{o}(1).$$

Note that $f=\hat{z}h(\hat{z})^{-1}$.  As in the proof of Theorem \ref{thm1}, we use Proposition \ref{prop}, \eqref{prop1}, to find an object isomorphic to $E_g$ which matches the object given in the theorem, completing the proof of Theorem \ref{thm2}.
\end{proof}

\subsection{Proof of Theorem \ref{thm3}}\label{subsec proof 3}

\begin{proof}
The calculations are virtually identical to the proof of Theorem \ref{thm2}, but the expressions are written in terms of $\z$ instead of $\hat{z}$, and $s-r$ instead of $r-s$.  Starting with \cite[Proposition 3.12]{bloch}, in our notation we have

\begin{equation*}
    \tz^2\nabla_{\tz}=\hat{z} \text{ and } \tz^{-1}=-\z^2\hat{\nabla}_{\z}.
\end{equation*}
Repeating the calculations of Theorem \ref{thm2} we conclude that
$$g(\z)=-\z^{-1}z = -\z^{-1}h(\z)^{-1}+\frac{\mathbb{Z}}{s-r} + \frac{s}{2(s-r)} + \underline{o}(1).$$
Note that $-\z^{-1}h(\z)^{-1}=f$.  As before, by considering an appropriate isomorphic object we eliminate the term with $\mathbb{Z}$, completing the proof of Theorem \ref{thm3}.
\end{proof}

\section{Comparison with previous results}\label{fang compare}

One notes that in \cite{fang}, Fang's Theorems 1, 2, and 3 look slightly different from those given in (respectively) our Theorems \ref{thm1}, \ref{thm2}, and \ref{thm3}.  We shall present a brief explanation for the equivalence of Fang's Theorem 1 and our Theorem \ref{thm1}.  One large difference in our methods is that Fang's calculations are split into a regular and irregular part, whereas we calculate both parts simultaneously.  We first verify the equivalence for the irregular part.

\subsection{Equivalence for the irregular part}\label{subsec equiv1}

Suppose $f$ in Theorem \ref{thm1} has zero regular part. In particular, this means that $f$ has no constant term.  Then with Fang's notation on the left and our notation on the right, we have the following relationships:
\[ t \text{  corresponds to  } z\]
\[ t' \text{  corresponds to  } \hat{z}\]
\[ t\partial_t(\alpha) \text{  corresponds to  } f\]
\[ (1/t')\partial_{(1/t')}(\beta)+\frac{s}{2(r+s)} \text{  corresponds to  } g\]

Using the correspondences above and equation (2.1) from Fang's paper, one can manipulate the systems of equations to see that the theorems coincide on the irregular part.

\subsection{Equivalence for the regular part}\label{subsec equiv2}

In \cite{fang}, the structure of the theorems is such that the calculation of the regular part is quite straightforward.  Using our theorems, however, the calculation of the regular part is hidden. To verify that the regular portion of our calculation matches up with the results from \cite{fang}, it suffices to prove the claim below.  We note that one can also calculate the regular part by using the global Fourier transform and meromorphic Katz extension; our proof is independent of that method.

\begin{claim}\label{claim}
	Let $f(z)=az^{-s/r}+\dots+b$ as in Theorem \ref{thm1} and $\mathcal{F}^{(0, \infty)}(E_f) = E_g$. Then $g$ will have constant term $\d{\left(\frac{r}{r+s}\right)b+\frac{s}{2(r+s)}}$.
\end{claim}

Before we prove Claim \ref{claim}, we first prove two lemmas regarding general facts about formal Laurent series and compositional inverses.

\begin{lemma}\label{Laurinv}
  Let $j(z)\in K_q$ with ord$(j)=\frac{p}{q}$, $p\in \bigz -\{0\}$ and $q>0$.  If $p>0$, then $j$ has a  formal compositional inverse $j^{\langle -1 \rangle}\in\Bbbk((z^{1/p}))$.   If $p<0$, then $j$ has a formal compositional inverse $j^{\langle -1 \rangle}\in\Bbbk((\zeta^{1/p}))$.
\end{lemma}

\begin{proof}
	Let  $h(z) = (z^{1/p}\comp j\comp z^q)(z)$.  Then $h(z)$ is a formal power series with no constant term and a nonzero coefficient for the $z$ term.  Such a power series will have a compositional inverse, call it $h^{\langle -1 \rangle}(z)$.  Then $j^{\langle -1 \rangle}(z):=(z^q\comp h^{\langle -1 \rangle}\comp z^{1/p})(z)$ will be a compositional inverse for $j$.
\end{proof}

\begin{remark}  Note that $h$ (and $h^{\langle -1 \rangle}$ as well) is not unique  since a choice of root of unity is made.  This will not affect our result, though, since $h^p$ and $(h^{\langle -1 \rangle})^q$ will be unique.
\end{remark}

\begin{lemma}\label{fgcoeffslemma}
	Let $\d{j(z)=az^{-(r+s)/r}+\dots+bz^{-1}+\underline{o}(z^{-1})}$, $j(z)\in K_r$, with $s$ a nonnegative integer and $r\in\mathbb{Z}^+$.  Then the coefficient for the $z^{-1}$ term of $j^{\langle -1 \rangle}(z)$ will be $\frac{br}{r+s}$.
\end{lemma}

\begin{proof}
	  Let $h(z) = (z^{-1/(r+s)}\comp j\comp z^r)(z)$.  Then $j(z^r)=h^{-(r+s)}$ and from the proof of Lemma \ref{Laurinv} we have
\begin{equation}\label{inverseeqs}
	 j^{\langle -1 \rangle}(z^{-(r+s)})=(h^{\langle -1 \rangle})^r.
\end{equation}
According to the Lagrange inversion formula, the coefficients of $h$ and $h^{\langle -1 \rangle}$ are related by
\begin{equation}\label{lif}
	(r+s)[z^{r+s}](h^{\langle -1 \rangle})^r=r[z^{-r}]h^{-(r+s)}
\end{equation}
where $[z^{r+s}](h^{\langle -1 \rangle})^r$ denotes the coefficient of the $z^{r+s}$ term in the expansion of $(h^{\langle -1 \rangle})^r$.  Substituting \eqref{inverseeqs} and $j(z^r)=h^{-(r+s)}$ into \eqref{lif} we conclude that
\begin{equation}\label{lifconc}
	[z^{r+s}]j^{\langle -1 \rangle}(z^{-(r+s)})=\frac{r}{r+s}[z^{-r}]j(z^r)
\end{equation}Since $[z^{-r}]j(z^r)=b$, the conclusion follows.
\end{proof}

\begin{proof}[Proof of Claim \ref{claim}]
    Given the notation used above for the Lagrange inversion formula, we can restate the claim as follows: if $[z^0]f=b$, then $[\z^0]g=\frac{br}{r+s}+\frac{s}{2(r+s)}$.\\
      Let $j(z)=-z^{-1}f$.  Then
    \begin{equation*}
    [z^{-1}]j=-[z^0]f=-b.
    \end{equation*}
      By \eqref{zisyseq1} we conclude that $\hat{z}=j(z)$, and by Lemma \ref{Laurinv} let $j^{\langle -1 \rangle}$ be the compositional inverse.  Then $j^{\langle -1 \rangle}(\hat{z})=z$.  From \eqref{zisyseq2} we have $g=-z\hat{z}+\frac{s}{2(r+s)}$, which implies that $-\hat{z}^{-1}(g-\frac{s}{2(r+s)})=j^{\langle -1 \rangle}(\hat{z})$.  This gives
\[ [\hat{z}^{-1}]j^{\langle -1 \rangle}=-[\hat{z}^0]g+\frac{s}{2(r+s)}
\]
or equivalently
      \begin{equation}\label{last}
    [\hat{z}^0]g=-[\hat{z}^{-1}]j^{\langle -1 \rangle}+\frac{s}{2(r+s)}.
    \end{equation}
    By Lemma \ref{fgcoeffslemma}, $[z^{-1}]j=-b$ implies that $[\hat{z}^{-1}]j^{\langle -1 \rangle}=\frac{-br}{r+s}$.  The result then follows from \eqref{last} and noting that $[\hat{z}^0]g=[\z^0]g$.
\end{proof}

\bibliographystyle{amsalpha}
\bibliography{dissreferences}


\end{document}